\documentclass[11pt,letterpaper,reqno]{amsart}
\usepackage{amssymb}
\usepackage{amsmath}
\usepackage{amsthm}
\usepackage{amsfonts}
\usepackage{enumitem}
\usepackage{booktabs}
\usepackage{graphicx}
\usepackage{xcolor}
\usepackage[T1]{fontenc}
\usepackage{doi}
\usepackage{comment}
\usepackage{hyperref}
\hypersetup{pdfstartview={FitH}}
\usepackage{bookmark}
\usepackage[capitalize,noabbrev]{cleveref}

\newtheorem{theorem}{Theorem}[section]
\newtheorem{lemma}[theorem]{Lemma}

\theoremstyle{definition}

\theoremstyle{plain}
\newtheorem*{openproblem}{Problem}

\DeclareMathOperator{\tr}{tr}

\begin{document}

\title[Fixed-volume Steklov product inequality]{A sharp fixed-volume product inequality for the first $N$ nonzero Steklov eigenvalues}

\author[H.~Zhang]{Haiqi Zhang}
\address{School of Mathematics, Shandong University, Jinan 250100, P. R. China}
\email{ZHQAQ2024@outlook.com}

\author[Q.~Tang]{Quanyu Tang}
\address{School of Mathematics and Statistics, Xi'an Jiaotong University, Xi'an 710049, P. R. China}
\email{tangquanyu827@gmail.com}

\author[Y.~Li]{Yanyang Li}
\address{School of Mathematics, Southeast University, Nanjing 211189, P. R. China}
\email{liyanyang1219@gmail.com}



\begin{abstract}
We prove a sharp fixed-volume product inequality for the first $N$ nonzero Steklov eigenvalues of bounded Lipschitz domains in $\mathbb R^N$. More precisely, if $N\ge2$ and $\Omega\subset\mathbb R^N$ is a bounded Lipschitz domain, then
$$
   \prod_{j=1}^N \sigma_j(\Omega)\le \frac{\omega_N}{|\Omega|},
$$
where $0=\sigma_0(\Omega)<\sigma_1(\Omega)\le\sigma_2(\Omega)\le\cdots$ are the Steklov eigenvalues of $\Omega$, and $\omega_N$ denotes the volume of the unit ball in $\mathbb R^N$. This extends the convex-domain theorem of Henrot, Philippin, and Safoui to arbitrary bounded Lipschitz domains, and in particular settles the remaining higher-dimensional case of a problem posed by Henrot.
\end{abstract}

\maketitle

\section{Introduction}

Throughout the paper, a domain means a connected open set. Let \(\Omega\subset\mathbb R^N\), \(N\ge2\), be a bounded Lipschitz domain. The Steklov eigenvalue problem is
\[
\begin{cases}
\Delta u=0 & \text{in }\Omega,\\
\frac{\partial u}{\partial \nu}=\sigma u & \text{on }\partial\Omega,
\end{cases}
\]
where \(\nu\) denotes the outer unit normal. Equivalently, the Steklov
eigenvalues are the eigenvalues of the Dirichlet-to-Neumann operator. Since
\(\Omega\) is Lipschitz, the spectrum is discrete and may be written as
\[
0=\sigma_0(\Omega)<\sigma_1(\Omega)\le \sigma_2(\Omega)\le\cdots\nearrow\infty,
\]
with eigenvalues repeated according to multiplicity; see, for instance,
\cite{GirouardPolterovich2017}.

Since \(\Omega\) is Lipschitz, \(\mathcal H^{N-1}(\partial\Omega)<\infty\); we denote this boundary measure by \(dS\). We shall use the following notation throughout the paper: \(d\theta\) denotes surface measure on \(\mathbb S^{N-1}\), and \(I=I_N\) denotes the \(N\times N\) identity matrix. For an invertible linear map \(C:\mathbb R^N\to\mathbb R^N\), we write
\(C\Omega=\{Cx:x\in\Omega\}\).

The eigenvalues satisfy the scaling law
\[
\sigma_j(t\Omega)=t^{-1}\sigma_j(\Omega),\qquad t>0.
\]
Thus the quantity
\[
\mathcal P_N(\Omega):=
|\Omega|\prod_{j=1}^N\sigma_j(\Omega)
\]
is invariant under dilations. For the Euclidean ball \(B_R:=\{x\in\mathbb R^N: |x|<R\}\), the first nonzero
Steklov eigenvalue has multiplicity \(N\), and
\[
\sigma_1(B_R)=\cdots=\sigma_N(B_R)=\frac1R.
\]
Since \(|B_R|=\omega_N R^N\), where \(\omega_N=|B_1|\), the ball satisfies
\begin{equation}\label{eq:ball-value}
\prod_{j=1}^N\sigma_j(B_R)=\frac{\omega_N}{|B_R|}.
\end{equation}

Sharp isoperimetric inequalities for Steklov eigenvalues have a long history
and are strongly sensitive to the choice of normalization and to the topology
of the domain. For recent surveys and developments on Steklov eigenvalue
problems and related geometric bounds, see
\cite{ColboisGirouardGordonSher2024,GirouardPolterovich2017}. In the plane, Weinstock proved that the disk maximizes
\(\sigma_1(\Omega)|\partial\Omega|\) among simply connected domains
\cite{Weinstock1954}. Hersch, Payne, and Schiffer later obtained product-type
inequalities for planar domains \cite{HerschPayneSchiffer1974}. Subsequent
work of Girouard and Polterovich established sharpness phenomena for higher
Steklov eigenvalues in the simply connected planar case
\cite{GirouardPolterovich2010HPS}. More recently, Girouard, Karpukhin, and
Lagac\'e solved the perimeter-normalized problem for planar domains: the sharp
upper bound for the \(k\)-th normalized Steklov eigenvalue is \(8\pi k\)
\cite{GirouardKarpukhinLagace2021}. This line of work is adjacent to the
present paper, but it concerns perimeter normalization rather than volume
normalization.

Under a volume constraint, Brock proved the sharp reciprocal-sum inequality
for the first \(N\) nonzero Steklov eigenvalues,
\[
\sum_{j=1}^{N}\frac{1}{\sigma_j(\Omega)}
\ge
N\left(\frac{|\Omega|}{\omega_N}\right)^{1/N},
\]
for bounded Lipschitz domains \(\Omega\subset\mathbb R^N\)
\cite{BrascoDePhilippisRuffini2012,Brock2001}. In particular, this implies
the Brock--Weinstock inequality that the ball maximizes the first nonzero
Steklov eigenvalue under a volume constraint. However, this reciprocal-sum
estimate does not directly control the product
\(\prod_{j=1}^N\sigma_j(\Omega)\). Related fixed-volume
optimization problems for Steklov eigenvalues and more general spectral
functionals were studied by Bogosel, Bucur, and Giacomini in a relaxed setting
\cite{BogoselBucurGiacomini2017}. In a different normalization, Fraser and
Schoen showed that in dimensions \(N\ge3\) the ball does not maximize
\(\sigma_1\) among contractible domains of fixed boundary measure
\cite{FraserSchoen2019}. These results illustrate the sensitivity of Steklov
optimization problems to the normalization and to the admissible class.

The product problem was settled in the convex class by Henrot, Philippin, and
Safoui. They proved that if \(\Omega\subset\mathbb R^N\) is convex and
\(\Omega^\ast\) is the ball with \(|\Omega^\ast|=|\Omega|\), then
\[
\prod_{j=1}^N\sigma_j(\Omega)
\le
\prod_{j=1}^N\sigma_j(\Omega^\ast),
\]
with equality only for the ball \cite{HenrotPhilippinSafoui2008}. Their proof
uses an isoperimetric inequality for a product of boundary moments of inertia,
whose proof is convex-geometric in dimensions \(N\ge3\). The natural question
whether the convexity assumption can be removed was posed explicitly in
Henrot's monograph \cite[Open problem 26]{Henrot2006}; it was also recalled in \cite[Remark~4.3]{BrascoDePhilippisRuffini2012}. In dimension \(N=2\),
however, the convexity assumption in this product inequality can already be
dropped; see the planar argument in
\cite[Section~5]{HenrotPhilippinSafoui2008} and the discussion in
\cite[Remark~4.3]{BrascoDePhilippisRuffini2012}. Thus the remaining open case
is the removal of convexity in dimensions \(N\ge3\).

\begin{openproblem}[Henrot]
Let \(N\ge3\), and let \(\Omega\subset\mathbb R^N\) be a bounded Lipschitz domain. Is it true that
\[
\prod_{j=1}^N\sigma_j(\Omega)
\le
\frac{\omega_N}{|\Omega|}?
\]
Equivalently, among bounded Lipschitz domains of prescribed volume,
does the ball maximize the product of the first \(N\) nonzero Steklov
eigenvalues?
\end{openproblem}

The purpose of the present paper is to answer this remaining
higher-dimensional problem affirmatively. Our main result is the following.

\begin{theorem}\label{thm:main}
Let \(N\ge2\), and let \(\Omega\subset\mathbb R^N\) be a bounded Lipschitz domain. Then
\begin{equation}\label{eq:main}
\prod_{j=1}^{N}\sigma_j(\Omega)\le \frac{\omega_N}{|\Omega|}.
\end{equation}
Equality is attained by balls. Conversely, if equality holds, then, after a
translation, \(\Omega\) agrees with a ball up to a set of Lebesgue measure zero.
\end{theorem}

In view of \eqref{eq:ball-value}, the constant in \eqref{eq:main} is optimal.

The main new ingredient is the following determinant estimate. After
translating the domain so that its boundary barycenter is at the origin, set
\[
M_\Omega:=\int_{\partial\Omega}x x^\top\,dS.
\]
We prove the sharp lower bound
\[
\det M_\Omega\ge \frac{|\Omega|^{N+1}}{\omega_N},
\]
for all bounded Lipschitz domains, with equality only for balls, up to null
sets. This replaces
the convexity-dependent boundary moment product estimate used in
\cite{HenrotPhilippinSafoui2008}. Combining this determinant estimate with the variational characterization of the Steklov eigenvalues, tested on the coordinate functions, gives the desired product bound.

The connectedness assumption is essential in the formulation above. If
disconnected domains are allowed and one skips all zero eigenvalues, then the
statement is false: the union of \(m\ge2\) equal balls of total volume \(V\)
has its first \(N\) positive Steklov eigenvalues all equal to \(1/r\), where
\(m\omega_N r^N=V\), and hence the product is \(m\omega_N/V\), larger than
the value \(\omega_N/V\) for a single ball.

The paper is organized as follows. In Section~\ref{sec:determinant-estimate}
we prove the determinant estimate for \(M_\Omega\). This is the main geometric
input of the proof. In Section~\ref{sec:steklov-product} we combine this
estimate with the standard Steklov variational principle for the coordinate
functions to prove Theorem~\ref{thm:main}.

\section{A determinant estimate}\label{sec:determinant-estimate}

In this section we prove the determinant estimate for \(M_\Omega\). We start
with a spherical integral estimate.

\begin{lemma}\label{lem:spherical-moment}
Let \(C\) be a positive definite symmetric \(N\times N\) matrix. For every \(0\le p\le N/2\),
\begin{equation}\label{eq:spherical-moment}
\frac{1}{N\omega_N}\int_{\mathbb S^{N-1}}(\theta^\top C\theta)^{-p}\,d\theta
\le
(\det C)^{-p/N}.
\end{equation}
Moreover, if \(0<p<N/2\), then equality in \eqref{eq:spherical-moment} holds if and only if \(C=cI\) for some \(c>0\).
\end{lemma}

\begin{proof}
Let \(d\mu=(N\omega_N)^{-1}d\theta\) be normalized surface measure on \(\mathbb S^{N-1}\). Then
\begin{align*}
\int_{\mathbb R^N}e^{-x^\top Cx}\,dx 
&=
\int_{\mathbb S^{N-1}} \left( \int_0^\infty e^{-r^2(\theta^\top  C \theta)} r^{N-1} dr \right) d\theta \\
&=\frac{1}{2}\int_{\mathbb S^{N-1}} (\theta^\top  C \theta)^{-N/2} d\theta\int_{0}^{\infty} e^{-u} u^{(N/2)-1} du \\
&=\frac12\Gamma\left(\frac N2\right)
\int_{\mathbb S^{N-1}}(\theta^\top C\theta)^{-N/2}\,d\theta.
\end{align*}
On the other hand, we have
\[
\int_{\mathbb R^N}e^{-x^\top Cx}\,dx 
=(\det C)^{-1/2} \prod_{i=1}^N \left( \int_{-\infty}^\infty e^{-y_i^2} dy_i \right)
=\pi^{N/2} (\det C)^{-1/2}.
\]
Applying the same identity with \(C=I\) yields
\[
\frac12\Gamma\left(\frac N2\right)N\omega_N=\pi^{N/2},
\]
and hence
\begin{equation}\label{eq:spherical-endpoint}
\int_{\mathbb S^{N-1}}(\theta^\top C\theta)^{-N/2}\,d\mu(\theta)
=
(\det C)^{-1/2}.
\end{equation}

Set $X(\theta)=(\theta^\top C\theta)^{-N/2}$ and $t=2p/N\in[0,1]$. Since \(s\mapsto s^t\) is concave on \((0,\infty)\) for \(0\le t\le1\), Jensen's inequality and \eqref{eq:spherical-endpoint} give
\[
\int_{\mathbb S^{N-1}}(\theta^\top C\theta)^{-p}\,d\mu(\theta)
=
\int_{\mathbb S^{N-1}}X(\theta)^t\,d\mu(\theta)
\le
\left(\int_{\mathbb S^{N-1}}X(\theta)\,d\mu(\theta)\right)^t
=
(\det C)^{-p/N}.
\]
This proves \eqref{eq:spherical-moment}.

If \(0<p<N/2\), then \(0<t<1\), and \(s\mapsto s^t\) is strictly concave. Equality in Jensen's inequality therefore forces \(X\) to be constant \(\mu\)-almost everywhere. By continuity, \(\theta^\top C\theta\) is constant on all of \(\mathbb S^{N-1}\). A quadratic form that is constant on the unit sphere is a scalar multiple of \(|\theta|^2\), so \(C=cI\) for some \(c>0\). The converse is immediate.
\end{proof}

We shall also use the following homogeneous version of the bathtub principle;
compare \cite[Theorem~1.14]{LiebLoss2001}. We include the short proof in order
to record the sharp constant and the equality case in the present form.

\begin{lemma}\label{lem:homogeneous-bathtub}
Let \(g:\mathbb R^N\to[0,\infty)\) be continuous, positive on \(\mathbb S^{N-1}\), and homogeneous of degree one. Put
\[
\kappa=|\{y\in\mathbb R^N:g(y)<1\}|.
\]
Then, for every measurable set \(E\subset\mathbb R^N\) with \(|E|<\infty\),
\begin{equation}\label{eq:homogeneous-bathtub}
\int_E g(y)\,dy
\ge
\frac{N}{N+1}\kappa^{-1/N}|E|^{(N+1)/N}.
\end{equation}
If \(|E|>0\), equality holds in \eqref{eq:homogeneous-bathtub} if and only if
\[
E=\{y:g(y)<r\}
\quad\text{up to a null set},
\qquad
r=\left(\frac{|E|}{\kappa}\right)^{1/N}.
\]
\end{lemma}

\begin{proof}
Let $|E|=W$. By continuity and positivity on the sphere, there are constants \(0<m\le M<\infty\) such that
\[
m|y|\le g(y)\le M|y|,
\qquad y\in\mathbb R^N.
\]
Hence \(0<\kappa<\infty\). Let \(K_t=\{y:g(y)<t\}\). Since \(g\) is homogeneous of degree one,
\[
K_t=tK_1,
\qquad
|K_t|=\kappa t^N,
\qquad t>0.
\]
Moreover, for every \(t>0\),
\[
|\{g=t\}|
\le
\lim_{\varepsilon\downarrow0}|\{t-\varepsilon<g<t+\varepsilon\}|
\le
\lim_{\varepsilon\downarrow0}\kappa\bigl((t+\varepsilon)^N-(t-\varepsilon)^N\bigr)
=0.
\]

If \(W=0\), there is nothing to prove. Assume \(W>0\) and set
\[
r=\left(\frac{W}{\kappa}\right)^{1/N},
\qquad
K=K_r.
\]
Then \(|K|=|E|=W\), so \(|E\setminus K|=|K\setminus E|\). If
\(\int_E g=+\infty\), then the desired inequality is immediate. Otherwise
all the following quantities are finite, since \(K\) is bounded and \(g\) is
bounded on \(K\). Using \(|E\setminus K|=|K\setminus E|\), we obtain
\begin{align*}
\int_E g-\int_K g
&=
\int_{E\setminus K}g-\int_{K\setminus E}g \\
&=
\int_{E\setminus K}(g-r)+\int_{K\setminus E}(r-g) \\
&\ge 0,
\end{align*}
because \(g\ge r\) on \(E\setminus K\) and \(g<r\) on \(K\setminus E\).
Thus \(K\) minimizes \(\int_E g\) among sets of measure \(W\).

It remains to compute \(\int_K g\). By scaling,
\begin{align*}
\int_K g
&=
r^{N+1}\int_{K_1}g \\
&=
r^{N+1}\int_0^1 |\{y\in K_1:g(y)>s\}|\,ds \\
&=
r^{N+1}\int_0^1 \kappa(1-s^N)\,ds \\
&=
\frac{N}{N+1}\kappa r^{N+1}
=
\frac{N}{N+1}\kappa^{-1/N}W^{(N+1)/N}.
\end{align*}
This proves \eqref{eq:homogeneous-bathtub}.

If equality holds, then equality must hold in the comparison with \(K\), hence
\[
\int_{E\setminus K}(g-r)\,dy
+
\int_{K\setminus E}(r-g)\,dy
=
0.
\]
Both integrands are nonnegative. Therefore \(g=r\) a.e. on \(E\setminus K\),
and \(g=r\) a.e. on \(K\setminus E\). The latter set is contained in
\(K=\{g<r\}\), hence it is null. The former is contained, up to a null set, in
the level set \(\{g=r\}\), which has measure zero. Hence \(E=K\) up to a null
set. The converse is immediate from the computation above.
\end{proof}

Next we prove an anisotropic weighted isoperimetric inequality.

\begin{lemma}\label{lem:anisotropic}
Let \(\Omega\subset\mathbb R^N\) be a bounded Lipschitz domain of positive measure, and let \(A\) be a positive definite symmetric \(N\times N\) matrix. Then
\begin{equation}\label{eq:anisotropic}
\int_{\partial\Omega}x^\top Ax\,dS
\ge
N\omega_N^{-1/N}(\det A)^{1/N}|\Omega|^{(N+1)/N}.
\end{equation}
\end{lemma}

\begin{proof}
Let \(C=A^{1/2}\), and define \(F(x)=|Cx|\,Cx\).
The vector field \(F\) is \(C^1\) on \(\mathbb R^N\), with \(DF(0)=0\). Moreover,
\[
|F(x)|=|Cx|^2=x^\top Ax.
\]
Hence, for a.e. \(x\in\partial\Omega\),
\[
F(x)\cdot\nu(x)\le |F(x)|=x^\top Ax.
\]
The divergence theorem for bounded Lipschitz domains gives
\begin{equation}\label{eq:div-start}
\int_{\partial\Omega}x^\top Ax\,dS
\ge
\int_{\partial\Omega}F\cdot\nu\,dS
=
\int_{\Omega}\operatorname{div}F\,dx.
\end{equation}

Set \(y=Cx\). A direct computation gives, for \(y\ne0\),
\[
\operatorname{div}_xF(x)
=
|y|\tr C+\frac{y^\top Cy}{|y|}.
\]
The right-hand side extends continuously to \(y=0\) by assigning value \(0\). Therefore
\begin{equation}\label{eq:change-var}
\int_{\Omega}\operatorname{div}F\,dx
=
(\det C)^{-1}\int_{C\Omega}g_C(y)\,dy,
\end{equation}
where
\[
g_C(y)=
\begin{cases}
|y|\tr C+\dfrac{y^\top Cy}{|y|}, & y\ne0,\\[6pt]
0, & y=0.
\end{cases}
\]
The function \(g_C\) is continuous, positive on \(\mathbb S^{N-1}\), and homogeneous of degree one.

Let
\[
K_C=\{y\in\mathbb R^N:g_C(y)<1\},
\qquad
\kappa_C=|K_C|.
\]
By Lemma~\ref{lem:homogeneous-bathtub}, for every measurable set \(E\subset\mathbb R^N\) with \(|E|=W\),
\begin{equation}\label{eq:bathtub}
\int_E g_C(y)\,dy
\ge
\frac{N}{N+1}\kappa_C^{-1/N}W^{(N+1)/N}.
\end{equation}

It remains to estimate \(\kappa_C\). In polar coordinates,
\begin{equation}\label{eq:kappa-polar}
\kappa_C
=
\frac1N\int_{\mathbb S^{N-1}}
(\tr C+\theta^\top C\theta)^{-N}\,d\theta.
\end{equation}
Let \(\lambda_1,\ldots,\lambda_N\) be the eigenvalues of \(C\). For each \(\theta\in\mathbb S^{N-1}\), the arithmetic-geometric mean inequality applied to the \(N+1\) positive numbers
\[
\lambda_1,\ldots,\lambda_N,\theta^\top C\theta
\]
gives
\begin{equation}\label{eq:amgm}
\tr C+\theta^\top C\theta
\ge
(N+1)(\det C)^{1/(N+1)}(\theta^\top C\theta)^{1/(N+1)}.
\end{equation}
Combining \eqref{eq:kappa-polar}, \eqref{eq:amgm}, and Lemma~\ref{lem:spherical-moment} with \(p=N/(N+1)\), which is admissible since \(N\ge2\), we obtain
\begin{align}
\kappa_C
&\le
\frac1N (N+1)^{-N}(\det C)^{-N/(N+1)}
\int_{\mathbb S^{N-1}}(\theta^\top C\theta)^{-N/(N+1)}\,d\theta \notag\\
&\le
\frac1N (N+1)^{-N}(\det C)^{-N/(N+1)}
N\omega_N(\det C)^{-1/(N+1)} \notag\\
&=
\frac{\omega_N}{(N+1)^N\det C}.
\label{eq:kappa-bound}
\end{align}

Now apply \eqref{eq:bathtub} to \(E=C\Omega\), whose volume is $|C\Omega|=(\det C)|\Omega|$. Using \eqref{eq:kappa-bound}, we get
\begin{align}
\int_{C\Omega}g_C(y)\,dy
&\ge
\frac{N}{N+1}\kappa_C^{-1/N} |C\Omega|^{(N+1)/N} \notag\\
&\ge
N\omega_N^{-1/N}(\det C)^{1/N}
(\det C)^{(N+1)/N}|\Omega|^{(N+1)/N}.
\label{eq:g-lower}
\end{align}
Combining \eqref{eq:div-start}, \eqref{eq:change-var}, and \eqref{eq:g-lower}, we find
\[
\int_{\partial\Omega}x^\top Ax\,dS
\ge
N\omega_N^{-1/N}(\det C)^{2/N}|\Omega|^{(N+1)/N}.
\]
Since \(\det A=(\det C)^2\), this is \eqref{eq:anisotropic}.
\end{proof}

We now derive the determinant estimate for \(M_\Omega\).

\begin{theorem}\label{thm:det-bound}
Let \(\Omega\subset\mathbb R^N\) be a bounded Lipschitz domain of positive measure. Define the matrix
\[
M_\Omega:=\int_{\partial\Omega}x x^\top\,dS.
\]
Then
\begin{equation}\label{eq:det-bound}
\det M_\Omega\ge \frac{|\Omega|^{N+1}}{\omega_N}.
\end{equation}
Equality is attained by balls centered at the origin. Conversely, if equality
holds, then \(\Omega\) agrees with a ball centered at the origin up to a set of
Lebesgue measure zero.
\end{theorem}

\begin{proof}
First check \(M_\Omega\) is positive definite. Indeed, if \(a\in\mathbb R^N\) and
\[
a^\top M_\Omega a
=
\int_{\partial\Omega}(a\cdot x)^2\,dS
=
0,
\]
then the trace of the affine function \(u(x)=a\cdot x\) vanishes on \(\partial\Omega\). Since \(\Omega\) is Lipschitz, this means \(u\in H^1_0(\Omega)\). The function \(u\) is weakly harmonic in \(\Omega\), so using \(u\) as a test function gives
\[
0=\int_\Omega |\nabla u|^2\,dx=|a|^2|\Omega|.
\]
Thus \(a=0\), and \(M_\Omega\) is positive definite.

Apply Lemma~\ref{lem:anisotropic} with \(A=M_\Omega^{-1}\). Since
\[
\int_{\partial\Omega}x^\top M_\Omega^{-1}x\,dS
=
\tr\left(M_\Omega^{-1}\int_{\partial\Omega}x x^\top\,dS\right)
=
\tr I
=
N,
\]
we get
\[
N
\ge
N\omega_N^{-1/N}(\det M_\Omega^{-1})^{1/N}|\Omega|^{(N+1)/N}.
\]
After rearranging,
\[
\det M_\Omega\ge \frac{|\Omega|^{N+1}}{\omega_N}.
\]

Assume now that equality holds in \eqref{eq:det-bound}. Set $A=M_\Omega^{-1}$ and $C=A^{1/2}$. Then the application of Lemma~\ref{lem:anisotropic} above has equal left- and
right-hand sides. Indeed, with \(F(x)=|Cx|Cx\), the proof of Lemma
\ref{lem:anisotropic} gives the chain
\begin{align*}
N
&=
\int_{\partial\Omega}x^\top Ax\,dS \\
&\ge
\int_{\partial\Omega}F\cdot\nu\,dS \\
&=
(\det C)^{-1}\int_{C\Omega}g_C(y)\,dy \\
&\ge
\frac{N}{N+1}(\det C)^{-1}\kappa_C^{-1/N}|C\Omega|^{(N+1)/N} \\
&\ge
N\omega_N^{-1/N}(\det C)^{2/N}|\Omega|^{(N+1)/N}.
\end{align*}
Since equality holds in \eqref{eq:det-bound}, the last expression is also
equal to \(N\). Consequently every inequality in the displayed chain is an
equality. In particular, equality holds in \eqref{eq:kappa-bound} and in the
application of Lemma~\ref{lem:homogeneous-bathtub} to \(E=C\Omega\).

We first show that \(C\) is scalar. The estimate \eqref{eq:kappa-bound}
is a two-step chain: the first inequality comes from the AM--GM estimate
\eqref{eq:amgm}, and the second from Lemma~\ref{lem:spherical-moment} with
\(p=N/(N+1)\). Since equality holds in \eqref{eq:kappa-bound}, equality must
hold in both steps. In particular, equality holds in Lemma
\ref{lem:spherical-moment} for \(p=N/(N+1)\). Since \(N\ge2\), we have
\[
0<\frac{N}{N+1}<\frac N2.
\]
The equality case in Lemma~\ref{lem:spherical-moment} therefore gives
\(C=cI\) for some \(c>0\).

For this scalar matrix \(C\), we have $g_C(y)=(N+1)c|y|$.
Equality in Lemma~\ref{lem:homogeneous-bathtub}, applied to \(E=C\Omega\), gives $C\Omega=\{y:g_C(y)<s\}$
up to a null set for some \(s>0\). Since \(g_C(y)=(N+1)c|y|\), this sublevel set is a ball centered at the origin. As \(C=cI\), \(\Omega\) itself agrees with a centered ball up to a null set. Conversely, a centered ball gives equality directly.
\end{proof}

\section{Proof of the Steklov product inequality}\label{sec:steklov-product}

We now prove Theorem~\ref{thm:main}.

\begin{proof}[Proof of Theorem \ref{thm:main}]
Let
\[
b=\frac{1}{|\partial\Omega|}\int_{\partial\Omega}x\,dS
\]
be the boundary barycenter of \(\Omega\). Replacing \(\Omega\) by
\(\Omega-b\), we may assume that
\[
\int_{\partial\Omega}x\,dS=0.
\]
This translation preserves both the volume and the Steklov eigenvalues.

For \(u\in H^1(\Omega)\), define the Steklov Rayleigh quotient by
\[
\mathcal R(u):=
\begin{cases}
\dfrac{\int_\Omega |\nabla u|^2\,dx}{\int_{\partial\Omega}u^2\,dS},
& \text{if } \int_{\partial\Omega}u^2\,dS>0,\\[10pt]
+\infty,
& \text{if } \int_{\partial\Omega}u^2\,dS=0.
\end{cases}
\]
For bounded Lipschitz domains, the trace embedding \(H^1(\Omega)\hookrightarrow L^2(\partial\Omega)\) is compact. The positive Steklov eigenvalues satisfy the following variational characterization; this is the standard Steklov variational principle with the zero eigenvalue removed (see, for instance, \cite[Theorem~7.1.9]{LMP23} and \cite[Eq.~(4.1.2)]{GirouardPolterovich2017}):
\[
\sigma_j(\Omega)
=
\min_{\substack{V\subset H^1(\Omega),\ \dim V=j\\
\int_{\partial\Omega}u\,dS=0\ \text{for every }u\in V}}
\ \max_{0\ne u\in V}\mathcal R(u),
\qquad j\ge1.
\]

Let
\[
L:=\{u_\alpha:\alpha\in\mathbb R^N\},
\qquad
u_\alpha(x):=\alpha\cdot x.
\]
The map \(\alpha\mapsto u_\alpha\) is injective, since
\(\nabla u_\alpha=\alpha\). Hence \(\dim L=N\). By the boundary centering,
\[
\int_{\partial\Omega}u_\alpha\,dS
=
\alpha\cdot\int_{\partial\Omega}x\,dS
=
0,
\]
so every function in \(L\) is orthogonal to constants in
\(L^2(\partial\Omega)\). Moreover,
\[
\|u_\alpha\|_{L^2(\partial\Omega)}^2
=
\int_{\partial\Omega}(\alpha\cdot x)^2\,dS
=
\alpha^\top M_\Omega\alpha.
\]
Since \(M_\Omega\) is positive definite by Theorem~\ref{thm:det-bound},
this quantity is positive whenever \(\alpha\ne0\). Thus every nonzero function
in \(L\) has a nonzero trace in \(L^2(\partial\Omega)\), and the Rayleigh
quotient is well defined on \(L\setminus\{0\}\).

For \(u_\alpha(x)=\alpha\cdot x\), we have
\[
\int_\Omega |\nabla u_\alpha|^2\,dx
=
|\Omega|\,|\alpha|^2.
\]
Together with the preceding boundary identity, this gives, for
\(0\ne\alpha\in\mathbb R^N\),
\begin{equation}\label{eq:coordinate-rayleigh}
\mathcal R(u_\alpha)
=
\frac{|\Omega|\,|\alpha|^2}{\alpha^\top M_\Omega\alpha}.
\end{equation}

Motivated by \eqref{eq:coordinate-rayleigh}, for \(j=1,\ldots,N\), define
\[
\rho_j
:=
\min_{\substack{S\subset\mathbb R^N\\ \dim S=j}}
\ \max_{0\ne\alpha\in S}
\frac{|\Omega|\,|\alpha|^2}{\alpha^\top M_\Omega\alpha}.
\]
Since \(M_\Omega\) is positive definite, the quotient appearing above is
positive and is bounded away from zero on \(\mathbb R^N\setminus\{0\}\). Hence
\(\rho_j>0\) for every \(j\). Moreover, the min--max definition immediately
gives
\[
\rho_1\le\cdots\le\rho_N.
\]
Indeed, if \(T\subset\mathbb R^N\) has dimension \(j+1\) and
\(S\subset T\) has dimension \(j\), then
\[
\max_{0\ne\alpha\in T}
\frac{|\Omega|\,|\alpha|^2}{\alpha^\top M_\Omega\alpha}
\ge
\max_{0\ne\alpha\in S}
\frac{|\Omega|\,|\alpha|^2}{\alpha^\top M_\Omega\alpha}
\ge
\rho_j.
\]
Taking the minimum over all such \(T\) yields \(\rho_{j+1}\ge\rho_j\).

We first compare the numbers \(\rho_j\) with the Steklov eigenvalues. If
\(S\subset\mathbb R^N\) has dimension \(j\), set
\[
V_S:=\{u_\alpha:\alpha\in S\}\subset L.
\]
Since the map \(\alpha\mapsto u_\alpha\) is injective, \(\dim V_S=j\).
Moreover, by the boundary centering,
\[
\int_{\partial\Omega}u_\alpha\,dS
=
\alpha\cdot\int_{\partial\Omega}x\,dS
=
0
\]
for every \(\alpha\in S\). Thus \(V_S\) is an admissible subspace in the
min--max formula for \(\sigma_j(\Omega)\). Hence
\[
\sigma_j(\Omega)
\le
\max_{0\ne u\in V_S}\mathcal R(u)
=
\max_{0\ne\alpha\in S}
\frac{|\Omega|\,|\alpha|^2}{\alpha^\top M_\Omega\alpha}.
\]
Taking the minimum over all \(j\)-dimensional subspaces
\(S\subset\mathbb R^N\), we obtain
\begin{equation}\label{eq:sigma-rho}
\sigma_j(\Omega)\le \rho_j,
\qquad j=1,\ldots,N.
\end{equation}

We now identify the numbers \(\rho_j\). Since \(M_\Omega\) is positive
definite, the symmetric matrix $|\Omega|\,M_\Omega^{-1}$ is positive definite. For \(0\ne\alpha\in\mathbb R^N\), put
\(\beta=M_\Omega^{1/2}\alpha\). Then
\[
\frac{|\Omega|\,|\alpha|^2}{\alpha^\top M_\Omega\alpha}
=
\frac{\beta^\top \bigl(|\Omega|\,M_\Omega^{-1}\bigr)\beta}
{\beta^\top\beta}.
\]
Since the map \(\alpha\mapsto\beta=M_\Omega^{1/2}\alpha\) is an isomorphism
of \(\mathbb R^N\), the Courant--Fischer min--max principle
\cite[Theorem~4.2.6]{HornJohnson2013} shows that
\(\rho_1,\ldots,\rho_N\) are precisely the eigenvalues of \(|\Omega|\,M_\Omega^{-1}\), listed in
nondecreasing order. Therefore,
\begin{equation}\label{eq:ritz-product}
\prod_{j=1}^N\rho_j
=
\det\bigl(|\Omega|\,M_\Omega^{-1}\bigr)
=
\frac{|\Omega|^N}{\det M_\Omega}.
\end{equation}

Therefore, using \eqref{eq:sigma-rho}, \eqref{eq:ritz-product} and Theorem~\ref{thm:det-bound},
\[
\prod_{j=1}^{N}\sigma_j(\Omega)
\le
\prod_{j=1}^{N}\rho_j
=
\frac{|\Omega|^N}{\det M_\Omega}
\le
\frac{|\Omega|^N}{|\Omega|^{N+1}/\omega_N}
=
\frac{\omega_N}{|\Omega|}.
\]
This proves \eqref{eq:main}.

If equality holds in \eqref{eq:main}, then the chain of inequalities above has equal endpoints. Hence
\[
\det M_\Omega=\frac{|\Omega|^{N+1}}{\omega_N},
\]
so equality holds in Theorem~\ref{thm:det-bound}. After the centering
translation, \(\Omega\) therefore agrees with a ball centered at the origin up
to a null set. Undoing the translation gives the stated rigidity assertion.
Conversely, balls give equality by \eqref{eq:ball-value}. Hence the proof is completed.
\end{proof}

\section*{Acknowledgments and AI disclosure}

During the preparation of this work the authors used ChatGPT (OpenAI) for exploratory mathematical discussions. After using this tool, the authors reviewed and edited the content as needed and take full responsibility for the content of the article.


\begin{thebibliography}{99}

\bibitem{BogoselBucurGiacomini2017}
B.~Bogosel, D.~Bucur, and A.~Giacomini,
\emph{Optimal shapes maximizing the Steklov eigenvalues},
SIAM J. Math. Anal. \textbf{49} (2017), no.~2, 1645--1680.

\bibitem{BrascoDePhilippisRuffini2012}
L.~Brasco, G.~De Philippis, and B.~Ruffini,
\emph{Spectral optimization for the Stekloff--Laplacian: the stability issue},
J. Funct. Anal. \textbf{262} (2012), no.~11, 4675--4710.

\bibitem{Brock2001}
F.~Brock,
\emph{An isoperimetric inequality for eigenvalues of the Stekloff problem},
Z. Angew. Math. Mech. \textbf{81} (2001), no.~1, 69--71.

\bibitem{ColboisGirouardGordonSher2024}
B.~Colbois, A.~Girouard, C.~Gordon, and D.~Sher,
\emph{Some recent developments on the Steklov eigenvalue problem},
Rev. Mat. Complut. \textbf{37} (2024), no.~1, 1--161.

\bibitem{FraserSchoen2019}
A.~Fraser and R.~Schoen,
\emph{Shape optimization for the Steklov problem in higher dimensions},
Adv. Math. \textbf{348} (2019), 146--162.

\bibitem{GirouardKarpukhinLagace2021}
A.~Girouard, M.~Karpukhin, and J.~Lagac\'e,
\emph{Continuity of eigenvalues and shape optimisation for Laplace and Steklov problems},
Geom. Funct. Anal. \textbf{31} (2021), 513--561.

\bibitem{GirouardPolterovich2010HPS}
A.~Girouard and I.~Polterovich,
\emph{On the Hersch--Payne--Schiffer estimates for the eigenvalues of the Steklov problem},
Funct. Anal. Appl. \textbf{44} (2010), no.~2, 106--117.

\bibitem{GirouardPolterovich2017}
A.~Girouard and I.~Polterovich,
\emph{Spectral geometry of the Steklov problem},
J. Spectral Theory \textbf{7} (2017), no.~2, 321--359.

\bibitem{Henrot2006}
A.~Henrot,
\emph{Extremum Problems for Eigenvalues of Elliptic Operators},
Frontiers in Mathematics, Birkh\"auser, Basel, 2006.

\bibitem{HenrotPhilippinSafoui2008}
A.~Henrot, G.~A. Philippin, and A.~Safoui,
\emph{Some isoperimetric inequalities with application to the Stekloff problem},
J. Convex Anal. \textbf{15} (2008), no.~3, 581--592.

\bibitem{HerschPayneSchiffer1974}
J.~Hersch, L.~E. Payne, and M.~M. Schiffer,
\emph{Some inequalities for Stekloff eigenvalues},
Arch. Rational Mech. Anal. \textbf{57} (1975), 99--114.

\bibitem{HornJohnson2013}
R.~A. Horn and C.~R. Johnson,
\emph{Matrix Analysis},
2nd ed.,
Cambridge University Press, Cambridge, 2013.

\bibitem{LiebLoss2001}
E.~H.~Lieb and M.~Loss,
\emph{Analysis},
2nd ed.,
Graduate Studies in Mathematics, vol.~14,
American Mathematical Society, Providence, RI, 2001.

\bibitem{LMP23}
M.~Levitin, D.~Mangoubi, and I.~Polterovich,
\emph{Topics in Spectral Geometry},
Graduate Studies in Mathematics, vol.~237,
American Mathematical Society, Providence, RI, 2023.

\bibitem{Weinstock1954}
R.~Weinstock,
\emph{Inequalities for a classical eigenvalue problem},
J. Rational Mech. Anal. \textbf{3} (1954), 745--753.
\end{thebibliography}
\end{document}